\documentclass[12pt,reqno]{amsart}

\usepackage{color}
\usepackage[active]{srcltx}
\usepackage{a4wide}
\usepackage{amssymb, amsmath, mathtools}
\usepackage{graphicx}
\usepackage{float}
\usepackage[active]{srcltx}
\usepackage[ansinew]{inputenc}

\usepackage{hyperref}

\theoremstyle{plain}
\newtheorem{theorem}{Theorem}[section]
\newtheorem{maintheorem}{Theorem}
\newtheorem{lemma}[theorem]{Lemma}

\theoremstyle{remark}

\newtheorem{remark}[theorem]{Remark}

\def\R{\ensuremath{\mathbb R}}
\def\N{\ensuremath{\mathbb N}}

\def\dist{\ensuremath{d}}

\def\lip{\operatorname{Lip}}

\newcommand{\qand}{\quad\text{and}\quad}

\numberwithin{equation}{section}

\begin{document}
\title[Equilibrium states for impulsive semiflows]{Equilibrium states for impulsive semiflows}

\author[J. F. Alves]{Jos\'{e} F. Alves}
\address{Jos\'{e} F. Alves\\ Centro de Matem\'{a}tica da Universidade do Porto\\ Rua do Campo Alegre 687\\ 4169-007 Porto\\ Portugal}
\email{jfalves@fc.up.pt} \urladdr{http://www.fc.up.pt/cmup/jfalves}

\author[M. Carvalho]{Maria Carvalho}
\address{Maria Carvalho\\ Centro de Matem\'{a}tica da Universidade do Porto\\ Rua do
Campo Alegre 687\\ 4169-007 Porto\\ Portugal}
\email{mpcarval@fc.up.pt}

\author[Jaqueline Siqueira]{Jaqueline Siqueira}
\address{Jaqueline Siqueira\\ Centro de Matem\'{a}tica da Universidade do Porto\\ Rua do
Campo Alegre 687\\ 4169-007 Porto\\ Portugal}
\email{jaqueline.rocha@fc.up.pt}

\date{\today}
\thanks{JFA and MC were partially supported by CMUP (UID/MAT/00144/2013), which is funded by FCT (Portugal) with national (MEC) and European structural funds through the programs FEDER, under the partnership agreement PT2020.
JFA was also partially supported by Funda\c c\~ao Calouste Gulbenkian.  JS was supported by CNPq-Brazil.}

\begin{abstract}
We consider impulsive semiflows defined on compact metric spaces and give sufficient conditions, both on the semiflows and the potentials, for the existence and uniqueness of equilibrium states. We also generalize the classical notion of topological pressure to our setting of discontinuous semiflows and prove a variational principle.
\end{abstract}

\maketitle

\setcounter{tocdepth}{2}

\tableofcontents 

\section{Introduction}

Impulsive dynamical systems may be interpreted as  suitable mathematical models of real world phenomena that display abrupt changes in their behavior, and are  described by three  objects: a continuous semiflow on a metric space $X$; a set $D \subset X$ where the flow experiments  sudden perturbations; and an impulsive function $I:D\to X$ which determines the change on a trajectory each time it collides with the impulsive set $D$.
See for instance reference~\cite{LBS89}, where one may find several examples of evolutive processes which are analyzed through differential equations with impulses.

For many years the achievements on the theory of impulsive dynamical systems concerned the behavior of trajectories, their limit sets and their stability; see e.g. \cite{K94a} and references therein. The first results on the ergodic theory of impulsive dynamical systems were established in \cite{AC14}, where
sufficient conditions for the existence of invariant  probability measures on the Borel sets were given.
Afterwards, it was natural to look  for some special classes of  invariant measures. So far, a useful approach has been to use potentials and finding equilibrium states. However, as the classical notion of topological entropy  requires continuity and impulsive semiflows exhibit discontinuities, it became necessary to introduce  a generalized concept of topological entropy, and this has been done in  \cite{ACV15}.  Moreover, it was proved that the new notion coincides with the classical one for continuous semiflows, and also a partial variational principle for impulsive semiflows: the topological entropy coincides with the supremum of the metric entropies of time-one maps.

Our aim in this paper was to extend the results of \cite{ACV15} in two directions. Firstly we establish a variational principle for a wide class of potentials; then we present sufficient conditions for the existence and uniqueness of equilibrium states for those potentials. Once more, due to the discontinuities of the impulsive semiflows, we needed to define a genera\-lized concept of topological pressure; and again we show that this new definition coincides with the classical one for continuous semiflows.

\subsection{Impulsive semiflows}\label{se.impulsive-semiflow}

Consider a compact metric space $(X,d)$, a continuous semiflow $\varphi:\R^+_0 \times X\to X$, a nonempty compact  set $D\subset X$ and a continuous map $I:D \to X$ such that $I(D)\cap D=\emptyset$. Under these conditions we say that $(X,\varphi,D,I)$ is an \emph{impulsive dynamical system}.
The first visit of each $\varphi$-trajectory to $D$ will be registered by the function $\tau_1:X\to~[0,+\infty]$, defined as
$$
\tau_1(x)=
\begin{dcases}
\inf\left\{t> 0 \colon \varphi_t(x)\in D\right\} ,& \text{if } \varphi_t(x)\in D\text{ for some }t>0;\\
+\infty, & \text{otherwise.}
\end{dcases}
$$
The \emph{impulsive trajectory} $\gamma_x$  and the subsequent \emph{impulsive times} $\tau_2(x),\tau_3(x),\dots$ (possibly finitely many) of a given point $x\in X$ are defined according to the following rules:
for $0\le t<\tau_{1}(x)$ we set  $\gamma_x(t)=\varphi_t(x).$
Assuming that $\gamma_x(t)$ is defined for $t<\tau_{n}(x)$ for some $n\ge 1$, we set
  $$\gamma_x(\tau_{n}(x))=I(\varphi_{\tau_n(x)-\tau_{n-1}(x)}(\gamma_x({\tau_{n-1}(x)}))).$$
  Defining the $(n+1)^{\text{th}}$ impulsive time of $x$ as
  $$\tau_{n+1}(x)=\tau_n(x)+\tau_1(\gamma_x(\tau_n(x))),$$
 for $\tau_n(x)<t<\tau_{n+1}(x)$, we set
         $$\gamma_x(t)=\varphi_{t-\tau_n(x)}(\gamma_x(\tau_n(x))).$$
We define the \emph{time duration} of the trajectory of $x$ as
 $\Upsilon(x)=\sup_{n\ge 1}\,\{\tau_n(x)\}.
 $
Since we are assuming  $I(D)\cap (D)=\emptyset$, it follows from \cite[Remark 1.1]{AC14} that we have  $\Upsilon(x)=\infty$ for all $x\in X$. Thus we have the impulsive trajectories defined for all positive times.
 This allows us to introduce the  \emph{impulsive semiflow} $\psi$ of an impulsive dynamical system $(X,\varphi, D, I)$ as
$$
\begin{array}{cccc}
        \psi:  &  \mathbb{R}^+_0 \times X & \longrightarrow &X \\
        & (t,x) & \longmapsto & \gamma_x(t), \\
        \end{array}$$
where $\gamma_x$ stands for the impulsive trajectory of $x$ determined by $(X,\varphi,D, I)$.
It was  proved in  \cite[Proposition 2.1]{B07} that $\psi$ is indeed a semiflow, though not necessarily continuous.

\begin{remark}\label{taueta}
It is known that the function $\tau_1$ is  lower semicontinuous on the set $X\setminus D$; see \cite[Theorem~2.7]{C04a}. Since we are assuming  that $I(D)\cap (D)=\emptyset$ and $I(D)$ is  compact, then there exists some $\eta>0$ such that
for all $x\in X$  and all $n\in \N$  we have
  $\tau_{n+1}(x)-\tau_n(x)\geq\eta$.
  \end{remark}

Now we state some  conditions about the continuous semiflow $\varphi$ on the sets $D$ and $I(D)$ which will be useful for the statements of our main results.
We define for $t>0$
\begin{equation}\label{eq.dxi}
D_t=\bigcup_{x\in D }\{\varphi_t(x) \colon  0<s<t\}.
\end{equation}
Given $\xi>0$, we say that   $\varphi$ is \emph{$\xi$-regular on  $D$} if
\begin{enumerate}
\item $D_t$ is an open set  for all $0<t\le \xi$;
\item if $\varphi_t(x)\in D_{\xi}$ for some  $x\in X$ and $t>0$, then $ \varphi_{s}(x)\in D$  for some $0\le s<t$.
\end{enumerate}
We say that $\varphi$ satisfies a \emph{$\xi$-half-tube condition} on a compact set $A\subset X$ if
\begin{enumerate}
\item $ \varphi_t(x) \in A \,\Rightarrow \,\varphi_{t+s}(x) \notin A$ for all $0 < s < \xi$;
\item $\{\varphi_t(x_1) \colon  0<t\le \xi\}\cap \{\varphi_t(x_2) \colon  0<t\le \xi\}=\emptyset$ for all $x_1,x_2\in A$ with $x_1\neq x_2$;
\item there exists $C>0$ such that, for all $x_1,x_2\in A$ with $x_1\neq x_2$, we have
$$0\leq t<s\leq \xi \quad \Rightarrow \quad d\,(\varphi_t(x_1), \varphi_t(x_2))\leq C\,d\,(\varphi_s(x_1), \varphi_s(x_2)).$$
\end{enumerate}
In our main results  we will assume that  $\varphi$ satisfies a $\xi$-half-tube condition on the compact sets $D$ and $I(D)$. In particular, the first condition in the definition of $\xi$-half-tube   for $A=D$ implies  that
$\tau_1(x)\ge \xi>0$ for all $ x\in D$.
Given  $\xi>0$  we define
$$X_\xi=X\setminus (D_\xi\cup D).$$
Since $D$ is compact, $I$ is continuous and $I(D)\cap D=\emptyset$, we may choose $\xi$ small enough so that $I(D)\cap D_\xi=\emptyset$. Therefore,   the set  $X_\xi$ is forward invariant under $\psi$, that is
\begin{equation}\label{eq.forinva}
\psi_t(X_\xi) \subseteq X_\xi,\quad \forall t \geq 0.
\end{equation}
For future use, we introduce the function
$$\tau^*_\xi:X_\xi \cup D \to [0,+\infty]$$
defined as
$$
\tau^*_\xi(x)=
\begin{cases}
\tau_1(x), &\text{if $x\in X_\xi$};\\
0, &\text{if $x\in D$}.
\end{cases}
$$
For latter reference we gather as (C1)-(C5) all the  properties that we need about impulsive dynamical systems in the list below. We assume that there exists $\xi_0>0$  such that for all $0<\xi<\xi_0$ we  have:
\begin{enumerate}
\item[(C1)]  $I:D \to X$ is Lipschitz  with $\lip(I)\le 1$ and $I(D)\cap D=\emptyset$;
\item[(C2)] $I(\Omega_\psi \cap D) \subset \Omega_\psi \setminus D$, where $\Omega_\psi$ denotes the set of non-wandering points of $\psi$;
\item[(C3)]  $\varphi$ is $\xi$-regular on   $D$;
\item[(C4)] $\varphi$ satisfies a $\xi$-half-tube condition on both $D$ and $I(D)$;
\item[(C5)] $\tau^*_\xi$ is continuous.

\end{enumerate}
Notice that conditions  (C3)-(C4) hold, for instance, when $\varphi$ is a $C^1$ semiflow on a manifold for which $D$ and $I(D)$  are submanifolds transversal to the flow direction.  Moreover, condition (C2) ensures that $\Omega_\psi \setminus D$ is invariant by $\psi$ (cf. \cite[Theorem B]{AC14}) and conditions (C2) and (C5) are essential to guarantee that $\mathcal{M}_\psi(X) \neq \emptyset$ (cf. \cite[Theorem A]{AC14}).

\subsection{Expansiveness}\label{se.expansive-specification}
Here we recall the classical definition of expansiveness for a continuous semiflow and introduce an adapted version for an impulsive semiflow.

\subsubsection*{Continuous semiflow}
Let $\varphi$ be a continuous semiflow on a   metric space  $(X,d)$.
We say that $\varphi$ is  \emph{expansive} on $X$ if for every $\delta>0$ there exists $\varepsilon>0$ such that if  $x, y \in X$ and a continuous map $s: \mathbb{R}_0^+ \rightarrow \mathbb{R}_0^+$ with $s(0)=0$ satisfy
$d\,(\varphi_t(x),\varphi_{s(t)}(y)) < \varepsilon $  
for all $t\ge 0$, then
$y = \varphi_t(x)$
for some $0<t < \delta$.

\subsubsection*{Impulsive semiflow} Let  $\psi$ be  the semiflow  of an impulsive dynamical system $(X,\varphi,D,I)$.
Given $\varepsilon>0$, consider $B_\varepsilon(D)$ the $\varepsilon$-neighborhood of $D$ in $X$.
We say that $\psi$ is  \emph{expansive} on $X$ if for every $\delta>0$ there exists $\varepsilon>0$ such that if  $x, y \in X$ and a continuous map $s: \mathbb{R}_0^+ \rightarrow \mathbb{R}_0^+$ with $s(0)=0$ satisfy
$d\,(\psi_t(x),\psi_{s(t)}(y)) < \varepsilon$  
for all $t\ge 0$
such that $\psi_t(x),\psi_{s(t)}(y)\notin B_\varepsilon(D)$, then
$y = \psi_t(x)$
for some $0<t < \delta$.

\subsection{Specification} Let $\psi$ be a semiflow  on a   metric space $(X,d)$. We say that  $\psi$ has the \emph{specification} property on $X$ if for all $\varepsilon>0$ there exist $L>0$ such that,
for any sequence $x_0,\dots, x_n$ of points in $X$ and any sequence $0\le t_0<\cdots< t_{n+1}$ such that
$t_{i+1}-t_i \geq L$ for all $0\le i\le n$, there are  $y \in X$ and    $r:\mathbb{R}^+_0 \rightarrow \mathbb{R}^+_0$  constant on each interval $[t_i, t_{i+1}[$ and satisfying
$$ r([t_0, t_1[)=0\quad \text{and}\quad
 |r([t_{i+1}, t_{i+2}[) - r([t_i, t_{i+1}[)| < \varepsilon,
$$
for which
$$d\,(\psi_{t+r(t)}(y),\psi_t(x_i)) < \varepsilon, \quad  \forall\, t \in [t_i, t_{i+1}[ \quad\forall \,\,0\le i \le n.$$
The specification is said to be \emph{periodic} if we can always choose $y$ periodic.

\subsection{Equilibrium states}\label{se.equ-state}

Let $\psi$ be a semiflow on a compact metric space $(X, d)$. In what follows we will denote by $\mathcal{M}_{\psi}^t(X)$ the set  probability measures defined on the $\sigma$-algebra of the Borel subsets of $X$ which are invariant under $\psi_t$, and  set $$\mathcal{M}_{\psi}(X) = \bigcap_{t\geq0}\,\mathcal{M}_{\psi}^t(X).$$
An \emph{equilibrium state} for a continuous potential $f :X\to \mathbb R$ is a probability measure $\mu_f \in \mathcal{M}_{\psi}(X)$ which maximizes the map
$$ \mathcal{M}_{\psi}(X)\ni\mu \quad  \longmapsto \quad h_\mu(\psi_1) + \int f\,d\mu,$$
where $h_\mu(\psi_1) $ stands for the \emph{metric entropy} of the time-one map of the semiflow $\psi$ with respect to the measure $\mu$.
One of the main issues concerning equilibrium states is to determine an appropriate space of potentials. We consider again  the   cases of continuous and  impulsive semiflows separately.

\subsubsection*{Continuous semiflow} Given a continuous semiflow $\varphi$ on $X$, denote by ${V}(\varphi)$ the space of continuous maps $f:X \rightarrow \mathbb{R}$ for which there are $K>0$ and $\varepsilon > 0$ such that for every $t>0$ we have
\begin{equation}\label{temK}
\left|\int_0^t f(\varphi_s(x))\,ds - \int_0^t f(\varphi_s(y))\,ds\right|< K,
\end{equation}
whenever
$$d\,(\varphi_s(x), \varphi_s(y)) < \varepsilon ,\quad   \forall  s  \in [0,t].$$
It was proved in \cite{F77} that each $f \in V(\varphi)$ has a unique equilibrium state if $\varphi$ is continuous and satisfies expansiveness and periodic specification. The same conclusion was obtained in~\cite{CT15} without the assumption of periodicity in the specification.

\subsubsection*{Impulsive semiflow} Consider  now $\psi$ as  the semiflow of an impulsive dynamical system $(X,\varphi,D,I)$. In this context we need to restrict the set of potentials for which we are going to find  an equilibrium state, introducing a slightly more demanding version of the space  of potentials.
We define $V^*(\psi)$ as the set of continuous maps $f:X \rightarrow \mathbb{R}$ for which
\begin{enumerate}
\item $f(x)=f(I(x))$ for all $x\in D$;
\item there are $K>0$ and $\varepsilon > 0$ such that for every $t>0$  we have
\begin{equation}\label{temK*}
\left|\int_0^t f(\psi_s(x))\,ds - \int_0^t f(\psi_s(y))\,ds\right|< K,
\end{equation}
whenever
$d\,(\psi_{s}(x), \psi_s(y)) < \varepsilon$ for all $ s \in [0,t]$ such that $\psi_s(x),\psi_s(y)\notin B_\varepsilon(D)$.
\end{enumerate}

For instance, constant potentials   belong to~$V^*(\psi)$.
The aim of our first result is to extend Franco's Theorem \cite{F77} on the existence and uniqueness of equilibrium states for   potentials in $V^*(\psi)$. As in \cite{F77}, to ensure uniqueness we need to assume some finite dimensionality condition on the metric space~$X$; see e.g.  \cite[Chapter~3]{E95}. Here we also need a uniform  control on the number of preimages under the impulsive function~$I$.

\begin{maintheorem}\label{te.existence_uniqueness}
Let $X$ be a compact 
metric space and $\psi$ the semiflow of    an impulsive dynamical system $(X,\varphi,D,I) $ for which  (C1)-(C5) hold. If  $\psi$   is expansive and has the periodic specification property in  $\Omega_\psi \setminus D$, then any potential $f \in V^*(\psi)$ has an 
equilibrium state. Moreover, if $\dim(X)<\infty$  and there is $k>0$ such that $\#I^{-1}(\{y\})\le k$ for every $y \in I(D)$, then the equilibrium state is unique.
 \end{maintheorem}

In particular, taking $f$  the null function,  we deduce that the impulsive semiflow $\psi$ has a  probability measure of maximum entropy, which in some cases is unique.

\subsection{Topological pressure}\label{se.top-pressure}
Here we briefly recall the classical definition of topological pressure for continuous semiflows (see  \cite{F77} for details) and   generalize this concept to impulsive semiflows.

\subsubsection*{Classical definition}\label{sse.topopological_pressure}
Let $(X,d)$ be a compact metric space and $\varphi:X\times \mathbb{R}_0^+ \rightarrow X$ be a continuous semiflow. Given $\varepsilon > 0$ and $t \in \mathbb{R}^+$, a subset $E$ of $X$ is said to be \emph{$(\varphi,\varepsilon,t)$-separated} if for any $x,y \in X$ with $x \neq y$ there is some $s \in \,[0,t]$ such that $d\,(\varphi_s(x), \varphi_s(y))>\varepsilon$.
Given $f:X \rightarrow \mathbb{R}$ a continuous potential, define
\begin{align*}
Z(\varphi,f,\varepsilon,t) &= \sup\,\,\left\{\sum_{x \in E}\, e^{\int_0^t f(\varphi_s(x))\,ds} \colon E\, \text{ is }\,(\varphi,\varepsilon,t)\text{-separated}\right\},\\
P(\varphi,f,\varepsilon) &= \limsup_{t \rightarrow +\infty}\,\frac{1}{t}\,\log\,Z(\varphi,f,\varepsilon,t).
\end{align*}
The \emph{topological pressure} of   $f$ with respect to $\varphi$ is defined as
$$P(\varphi,f) = \lim_{\varepsilon \rightarrow 0}\,P(\varphi,f,\varepsilon).$$

\subsubsection*{New definition}\label{sse:new-definition}

Let $(X,d)$ be a compact metric space and $\psi: \R^+_0 \times X\to X$ a semiflow (possibly not continuous).
 Consider a function $T$ assigning to each $x\in X$ a  sequence $(T_n(x))_{n\in A(x)}$ of nonnegative numbers, where
either $A(x)=\N$ or $A(x)=\{1,\dots,\ell\}$ for some $\ell\in \N$.
We say that $T$ is \emph{admissible}  if there exists $\eta>0$ such that
for all $x\in X$  and all $n\in \N$ with $n+1\in A(x)$ we have
\begin{enumerate}
\item   $T_{n+1}(x)-T_n(x)\geq\eta$;
\item
$\displaystyle T_n(\psi_t(x))=
\begin{cases}T_n(x)-t, & \text{if $T_{n-1}(x)< t < T_n(x)$};\\
T_{n+1}(x),& \text{if $t=T_n(x)$}.
\end{cases}
$
\end{enumerate}
For each   admissible function $T$, $x\in X$, $t>0$ and $0<\delta<\eta/2$, we define
$$J_{t,\delta}^T(x) = [0,t]\setminus\bigcup_{j\in A(x)}]T_j(x)-\delta,T_j(x)+\delta[.$$
Observe that $ J_{t,\delta}^T(x)= [0,t]$, whenever $T_1(x)>t$.
Given $\varepsilon > 0$ and $t \in \mathbb{R}^+$, we say that  $E\subset X$ is  \emph{$(\varphi,\delta,\varepsilon,t)$-separated} if for any $x,y \in X$ with $x \neq y$ there is some $s \in J_{t,\delta}^T(x) $ such that $d\,(\varphi_s(x), \varphi_s(y))>\varepsilon$.
Given a continuous potential $f:X \rightarrow \mathbb{R}$, we define
\begin{align*}
Z^T(\varphi,f,\delta,\varepsilon,t) &= \sup\,\,\left\{\sum_{x \in E}\, \exp\left({\int_0^t f(\varphi_s(x))\,ds}\right) \colon E\, \text{ is finite and }\, (\varphi,\delta,\varepsilon,t)\text{-separated}\right\},\\
P^T(\varphi,f,\delta,\varepsilon) &= \limsup_{t \rightarrow +\infty}\,\frac{1}{t}\,\log\,Z_t^T(\varphi,f,\delta,\varepsilon),\\
P^T(\varphi,f,\delta) &= \lim_{\varepsilon \rightarrow 0}\,P^T(\varphi,f,\delta,\varepsilon).
\end{align*}
Finally, the \emph{$T$-topological pressure} of  $f$ with respect to $\varphi$ is defined as
$$P^T(\varphi,f) = \lim_{\delta \rightarrow 0}P^T(\varphi,f,\delta).$$
Notice that, as in the classical case, the \emph{$T$-topological pressure}  is well defined, because
\begin{enumerate}
\item if $0 < \varepsilon_1 < \varepsilon_2$, then $Z^T(\varphi,f, \varepsilon_1,\delta,t)\geq Z^T(\psi,f, \varepsilon_2,\delta,t)$;
\item if $0 < \varepsilon_1 < \varepsilon_2$, then $P^T(\psi,f,\varepsilon_1,\delta) \geq P^T(\psi,f,\varepsilon_2,\delta)$; 
\item if $0 < \delta_1 < \delta_2$, then $ P^T(\psi,f,\delta_1)\ge P^T(\psi,f,\delta_2)$.
\end{enumerate}
The next result shows that for continuous semiflows  the classical and new notions of topological pressure coincide.

\begin{maintheorem}\label{te.coincide} Let $X$ be a compact metric space, $\varphi$  a continuous semiflow on $X$ and $T$ an admissible function. If  $f:X \rightarrow \mathbb{R}$ is a continuous potential, then $P^T(\varphi,f)=P(\varphi,f).$
\end{maintheorem}

The previous result motivates our definition of topological pressure for an impulsive semiflow. First of all observe   that given $\psi$ the impulsive semiflow of an impulsive dynamical system  $(X,\varphi,D, I)$, the function $\tau$ assigning to each point in $X$ its sequence of impulsive times is admissible (recall Remark~\ref{taueta}). Therefore, we may define the \emph{topological pressure} of a potential $f:X\to \mathbb R$ with respect to an impulsive semiflow $\psi$ as $P^\tau(\psi,f)$.
In the sequel  we establish a \emph{variational principle} which generalizes \cite[Theorem~C]{ACV15}. Actually, for the particular choice of $f=0$  the next result  gives  \cite[Theorem~C]{ACV15}.

\begin{maintheorem}\label{te.pressure}
Let $X$ be a compact 
metric space and $\psi$ the semiflow of    an impulsive dynamical system $(X,\varphi,D,I) $ for which  (C1)-(C5) hold.
Then for any potential  $f\in  V^*(\psi)$ we have
$$P^\tau(\psi,f)=\sup_{\mu\in \mathcal{M}_{\psi}(X)} \left\{ h_\mu(\psi_1) + \int f\,d\mu\right\},$$
where  $\psi_1$ stands for the time one map of the semiflow $\psi$.
\end{maintheorem}

Under expansivity and specification assumptions, it follows from of Theorem A that  the supremum in this last result  is attained.  As $\Omega_{\psi_1} \subseteq \Omega_\psi$ and for any probability measure $\mu \in \mathcal{M}_{\psi}(X)$ we have $\mu(D)=0$ (see \cite[Lemma 4.7]{AC14}), then
$$h_\mu(\psi_1) + \int f\,d\mu= h_\mu({\psi_1}|_{\Omega_\psi\setminus D}) + \int_{\Omega_\psi\setminus D} f\,d\mu,$$
for any $\mu \in \mathcal{M}_{\psi}(X)$, and so
it follows from Theorem~\ref{te.pressure} that
$$P^\tau(\psi,f)= P^\tau(\psi|_{\Omega_\psi\setminus D},f|_{\Omega_\psi\setminus D}).
$$

\section{Classical and new pressure}\label{se.proof-Theorem C}
Here we prove that the modified definition of topological pressure coincides with the classical one for continuous semiflows and continuous potentials defined on compact metric spaces, thus proving Theorem~\ref{te.coincide}.

Let $(X,d)$ be a compact metric space, $\varphi$  a continuous semiflow on $X$ and $f:X \rightarrow \mathbb{R}$  a continuous potential. Given  $T$  admissible, consider $\eta>0$ as  in the definition of an admissible function
and fix constants $0<\delta<\eta/2$, $\varepsilon> 0$ and $t>0$. Notice that for every $x\in X$ we have
$$Z^T(\varphi,f,\delta,\varepsilon,t) \leq Z(\varphi,f,\varepsilon,t),$$
and so
$$P^T(\varphi,f)\leq P(\varphi,f).$$
We will now prove the reverse inequality. We start by stating  a useful lemma whose proof can be found in~\cite[Lemma~2.1]{ACV15}.

\begin{lemma}\label{le.alfabeta} Let $\varphi$ be a continuous semiflow on a compact metric space~$X$. For each $\varepsilon>0$ there is $\alpha>0$ such that $\dist(\varphi_s(x),\varphi_u(x))<\varepsilon$ for all $x\in X$ and  $s,u\geq 0$ with $|s-u|<\alpha$.
\end{lemma}

Consider an arbitrary  $\varepsilon>0$. By Lemma~\ref{le.alfabeta} there exists $\alpha>0$ such that for all $z\in X$ and all $s,u>0$ with $|s-u|<\alpha$, we have
\begin{equation}\label{eq.phiz}
 \dist(\varphi_s(z),\varphi_u(z))<\varepsilon/4.
\end{equation}
Hence, if $x,y\in X$ and 
$s\geq0$ satisfy
\begin{equation}\label{eq.phit}
\dist(\varphi_s(x),\varphi_s(y))>\varepsilon,
\end{equation}
then, for every $u\in (s-\alpha,s+\alpha)$, we get
$$\dist(\varphi_s(x),\varphi_s(y))\le \dist(\varphi_s(x),\varphi_u(x))+\dist(\varphi_u(x),\varphi_u(y))+\dist(\varphi_u(y),\varphi_s(y))$$
which, together with \eqref{eq.phiz}  and \eqref{eq.phit}, implies
\begin{equation}\label{eq:contunif}
\dist(\varphi_u(x),\varphi_u(y)) >\varepsilon/2.
\end{equation}
Consider  now $E\subseteq X$ being $(\varphi, t, \varepsilon)$-separated. As $\varphi$ is continuous, the set $E$ is finite. By definition, for every $x,y\in E$, $x\ne y$, there exists $s\in [0,t]$ such that
$$\dist(\varphi_s (x),\varphi_s (y))\geq\varepsilon.$$
Choose $0<\delta<\min\{\eta, \alpha/2,\varepsilon\}$ and $0<{\varepsilon'}<\varepsilon/2$. By \eqref{eq:contunif}, if  $u\in (s-2\delta,s+2\delta)$, then
$$\dist(\varphi_u(x),\varphi_u(y))>\varepsilon/2 >{\varepsilon'}.$$
If $s\in\,J_{t,\delta}^T(x)$ for some $t>0$, then $y\notin  B^T(x,\varphi, \delta, {\varepsilon'},t)$, where
$$B^T(x,\varphi, \delta, {\varepsilon'},t)=\left\{z\in X:\dist(\psi_s (x),\psi_s (z))< {\varepsilon'},\: \forall s\in J_{t,\delta}^T(x) \right\}.$$
Otherwise,  $J_{t,\delta}^T(x) \cap(s-2\delta,s+2\delta)\ne\emptyset$, and then $y\notin  B^T(x,\varphi, \delta,{\varepsilon'},t)$. So, $E$ is $(\varphi,T, \delta, {\varepsilon'},t)$-separated. Hence
$$Z(\varphi, \varepsilon, t)\leq Z^T(\varphi, \delta, {\varepsilon'},t), $$
and so
$$\frac1{t}\log Z(\varphi, f,\varepsilon, t) \leq \frac1{t}\log  Z^T(\varphi, f,\delta, {\varepsilon'},t).$$
Taking the upper limit as $t\to +\infty$, we get
$$P(\varphi,f,\varepsilon)\leq P^T(\varphi,f, \delta, {\varepsilon'}).$$
Now,  taking $\varepsilon'\to0$ we obtain
$$P(\varphi,f,\varepsilon)\leq P^T(\varphi,f,\delta).$$
Noticing that 
$\delta\to 0$ when $\varepsilon\to 0$, we have
$$P(\varphi,f) \leq P^T(\varphi,f).$$
This finishes the proof of Theorem~\ref{te.coincide}.

\section{Refinements and semiconjugacies}\label{se.refinement}

Given $T$ and $T'$ two admissible functions in $X$, we say that~$T'$ \emph{refines} $T$, and write $T'\succ T$, if for all $x\in X$ and $n\in\N$ there exists $m=m(n,x)\in\N$ such that $T_n(x)=T'_m(x)$. Our new concept of topological pressure is monotone with respect to the refinement of admissible functions and invariant by semiflow equivalences that respect the fixed admissible functions. The proofs of these properties are given in the two lemmas below and differ only in minor details from  analogous results in \cite[Section~2.2]{ACV15}.

\begin{lemma}\label{le.refine_invariance} Let $\psi$ be a semiflow on $ X$. If $T$ and $T'$ are admissible functions such that $T'\succ T$, then $P^T(\psi, f)\geq P^{T'}(\psi,f)$ for any  continuous potential  $f:X \rightarrow \mathbb{R}$.
\end{lemma}
\begin{proof} Given $\varepsilon ,t>0$, $0<\delta <\eta/2$ and a finite $(\psi,T',   \delta, \varepsilon, t)$-separated subset $E$, as $T'\succ T$, the set $E$ is a $(\psi,T, \delta ,  \varepsilon, t)$-separated as well. Therefore
$$Z^{T'}(\psi, \delta,\varepsilon,t)\leq Z^T(\psi, \delta,\varepsilon,t),$$
which clearly yields our conclusion.
\end{proof}

Given two semiflows $\psi:\R^+_0\times X \rightarrow X$ and $ {\psi}':\R^+_0\times {X}' \rightarrow {X}'$, acting on the metric spaces $(X,d)$ and $( {X}',d')$, and two admissible functions $T$ and $T'$ defined on $X$ and $X'$, respectively, we say that a uniformly continuous surjective map $h:X \rightarrow X'$ is a \emph{$(T,T')$-semiconjugacy} between $\psi$ and $ \psi'$ if
\begin{enumerate}
\item $ {\psi}'_t\circ h=h\circ \psi_t$ for all $ t\geq 0$;
\item $ {T}'(h(x))=T(x)$ for all $x \in X.$
\end{enumerate}

\begin{lemma}\label{le.invariant} Let $h:X\to X'$ be a $(T,T')$-semiconjugacy between the semiflows $\psi$ and ${\psi'}$ on $X$ and ${X'}$, respectively,   such that the pre-image of each point under $h$ is finite. Then $P^T(\psi,f\circ h)\ge P^{{T'}}({\psi'},f)$ for any  continuous potential $f:{X'} \rightarrow \mathbb{R}$.
\end{lemma}

\begin{proof}
Let $\psi:\R^+_0\times X \rightarrow X$ and ${\psi'}:\R^+_0\times {X'} \rightarrow {X'}$ be two semiconjugate semiflows and $h$ be such a semiconjugacy. As $h$ is uniformly continuous, given $\varepsilon >0$ there exists $\varepsilon'>0$ such that for $a,b \in X$
$$d(a,b)<\varepsilon' \quad \Rightarrow\quad {d'}(h(a), h(b))<\varepsilon .$$
Fix $t>0$ and $0<\delta<\eta/2$, and consider a finite $({\psi'},{T'}, t, \varepsilon,\delta)$-separated set $B\subseteq {X'}$. Then $A=h^{-1}(B)$ is finite, although it may have a cardinal bigger or equal than the one of $B$. Moreover, $A$ is a $(\psi,T, t, \varepsilon',\delta)$-separated set of $X$. Indeed, for all $a,b \in A$, there are $t_n \in \,J_{t,\delta}^{{T'}}(h(a))$ and $s_n \in \,J_{t,\delta}^{{T'}}(h(b))$ such that
$${d'}({\psi'}_{t_n}(h(a)), {\psi'}_{t_n}(h(b)))\geq \varepsilon \quad \text{and} \quad {d'}({\psi'}_{s_n}(h(a)), {\psi'}_{s_n}(h(b)))\geq \varepsilon$$
that is,
$${d'}(h\circ \psi_{t_n}(a), h\circ \psi_{t_n}(b))\geq \varepsilon \quad \text{and} \quad {d'}(h\circ \psi_{s_n}(a), h\circ \psi_{s_n}(b))\geq \varepsilon.$$
Therefore,
$$d (\psi_{t_n}(a), \psi_{t_n}(b))\geq \varepsilon' \quad \text{and} \quad d(\psi_{s_n}(a), \psi_{s_n}(b))\geq \varepsilon'.$$
Taking into account that, by definition of semiconjugacy, $t_n \in \,J_{t,\delta}^T(a)$ and $s_n \in \,J_{t,\delta}^T(b)$, we deduce that
$$Z^T(\psi, f\circ h,\delta,\varepsilon',t)\geq Z^{{T'}}({\psi'}, f,\delta,\varepsilon,t).$$
Noticing that  $\varepsilon' \rightarrow 0$ when $\varepsilon \rightarrow 0$, we finally conclude our proof.
\end{proof}

\section{Quotient dynamics}\label{sec.quotient}

Given  an impulsive dynamical system $(X,\varphi,D,I)$, consider  the equivalence relation $\sim$ on $X$ given by
 \begin{equation}\label{eq.tilde}
x\sim y\quad \Leftrightarrow \quad x=y, \quad y=I(x),\quad x=I(y)\quad \text{or}\quad I(x)=I(y).
\end{equation}
Let  $\widetilde X$ denote the quotient space,  $\tilde x$   the equivalence class of a point   $x\in X$ and  $\pi:X\to \widetilde X$    the natural projection. It follows from \cite[Lemma~4.1]{AC14} that $\widetilde X$ is a metrizable space. Actually, if $d$ denotes the metric on $X$,  a metric $\tilde d$ on $\widetilde X$ that induces the quotient topology is given by
$$ \tilde d \, (\tilde x,\tilde y) = \inf \, \{d\,(p_1, q_1)+ d\,(p_2, q_2)+\cdots +d\,(p_n,q_n)\},$$
where $p_1, q_1, \dots, p_n, q_n$ is any chain of points in $X$ such that
$ p_1 \sim x$, $q_1 \sim p_2$, $q_2 \sim p_3$, ... $q_n \sim y$.
In particular, we have for all $x,y\in X$
\begin{equation}\label{eq.dtil}
\tilde d \, (\tilde x,\tilde y) \leq d\,(x,y).
\end{equation}
In general, an inequality in the opposite direction  is much more complicate. In the case that $I$ does not expand distances we have the following result whose proof may be found in \cite[Lemma 4.1]{ACV15}.

\begin{lemma}\label{le.chain}
If $\lip(I)\le 1$, then for all $\,\tilde x, \tilde y \in \pi(X)$ there exist $p, q \in X$ such that
$p \sim x$, $  q \sim y$ and $ d(p,q) \leq 2\,\tilde d\,(\tilde x, \tilde y).$
\end{lemma}

Now take $\xi>0$ such that conditions (C1)-(C5) hold. Since $I(D)\cap D = \emptyset$, then each point in the set $X_\xi=X\setminus (D_\xi\cup D)$ has a representative of its equivalence class in $X\setminus D_\xi $. This implies that
\begin{equation}\label{eq.xxi}
\pi(X_\xi)=\pi(X\setminus D_\xi),
\end{equation}
and by the $\xi$-regular condition (C3) this is a compact set. As we are assuming that $\varphi$ satisfies a $\xi$-half-tube condition, then $X_\xi$ is invariant under~$\psi$; recall \eqref{eq.forinva}. With no risk of confusion, the restriction of $\psi$ to $X_\xi$ will still be denoted by~$\psi$.

Given $x,y\in X_\xi$ we have $x\sim y$ if and only if $x=y$. So $\pi\vert_{X_\xi}$ induces a continuous bijection from $X_\xi$ onto the set
$$\widetilde X_\xi=\pi(X_\xi)$$
that we shall denote by $H$. The map $H$ allows us to introduce a semiflow
$\widetilde\psi$ on $ \widetilde X_\xi $
given by
\begin{equation}\label{def.psi}
\widetilde\psi(t,\tilde x)=H \circ\psi(t,x),
\end{equation}
for
all $x\in X_\xi$ and   $t\ge 0$. Since the impulsive semiflow $\psi$ satisfies conditions (C1) and (C5), it follows from  \cite[Lemma 4.2]{ACV15}  that the semiflow $\widetilde\psi$ is continuous. Moreover, $H$ gives a semiconjugacy between the semiflows, i.e.
\begin{equation}\label{eq.conjuga2}
\widetilde\psi_t( H (x) )= H(\psi_t (x))
\end{equation}
for
all $x\in X_\xi$ and  $t\ge 0$.

 We are looking for measures of maximal entropy for the impulsive semiflow $\psi$ and, more generally, equilibrium states for  potentials in $V^*(\psi)$. The strategy is to bring to $X$, via $H$, equilibrium states for the continuous quotient dynamics $\widetilde{\psi}$ defined on the compact quotient metric space $\widetilde{X}_\xi$. For this purpose, we will need to carry the expansiveness and periodic specification properties from the impulsive semiflow $\psi$ to the quotient semiflow~$\widetilde \psi$. 
However, 
condition (C3) is incompatible with the periodic specification property for $\psi$ in~$X_\xi$. This is why we will restrict the impulsive semiflow $\psi$ to $\Omega_\psi \setminus D$, a set that is contained in $X_\xi$ for every $0<\xi < \xi_0$. Recall that, as $\psi$ is not continuous, its non-wandering set $\Omega_\psi$, although closed, may be not $\psi$-invariant. Yet, condition (C2) guarantees that $\Omega_\psi \setminus D$ is $\psi$-invariant (see \cite[Theorem B]{AC14}) and, moreover, that $\pi(\Omega_\psi \setminus D)=\pi(\Omega_\psi)$. 
In what follows we will consider  $\widetilde \psi$ restricted to the compact set
$$\widetilde Y=\pi(\Omega_\psi \setminus D).$$

\begin{lemma}\label{le.properties}
Let $\psi$ be an impulsive semiflow. 
\begin{enumerate}
\item  If $\psi:\mathbb{R}^+_0 \times (\Omega_\psi \setminus D) \rightarrow \Omega_\psi \setminus D$ is expansive, then $\widetilde{\psi}:\mathbb{R}^+_0 \times \widetilde Y\rightarrow \widetilde Y$ is expansive.
\item  If $\psi:\mathbb{R}^+_0 \times (\Omega_\psi \setminus D) \rightarrow \Omega_\psi \setminus D$ has the (periodic) specification property, then  $\widetilde{\psi}:\mathbb{R}^+_0 \times \widetilde Y \rightarrow \widetilde Y$ has the (periodic) specification property.
\end{enumerate}
\end{lemma}

\begin{proof}
(1)
Suppose that $\psi$ is expansive on  $\Omega_\psi \setminus D$. Given  $\delta>0$, consider $\varepsilon>0$ associated to $\delta$ as in the definition of expansiveness for the  impulsive semiflow~$\psi$.
Let  $\tilde x, \tilde y \in \widetilde Y$ and a continuous map $s: \mathbb{R}_0^+ \rightarrow \mathbb{R}_0^+$ with $s(0)=0$ satisfying
$\tilde d(\widetilde\psi_t(\tilde x),\widetilde \psi_{s(t)}(\tilde y)) < \varepsilon/2$  
for all $t\ge 0$.
Letting $x=H^{-1}(\tilde x)$ and $y=H^{-1}(\tilde y)$, we have for all $t\ge 0$
 \begin{equation*}
 \psi_t(x)=H^{-1}(\widetilde \psi_t(\tilde x))\qand \psi_{s(t)}(y)=H^{-1}(\widetilde \psi_{s(t)}(\tilde y)).
\end{equation*}
Recall   that the equivalence classes of $\psi_s(x)$ and $\psi_s(y)$ are reduced to a single point. Then, if $\psi_t(x),\psi_{s(t)}(y)\notin B_\varepsilon(D)$, it follows from Lemma~\ref{le.chain} that
$d\,(\psi_t( x), \psi_{s(t)}( y)) < \varepsilon.$
Therefore,
$ y =  \psi_t( x)$
for some $0<t < \delta$. This implies that
$\tilde y = \widetilde \psi_t(\tilde x)$, and so $\widetilde\psi$ is expansive.

(2)
Assume now that $\psi$ has the (periodic) specification property on  $\Omega_\psi \setminus D$.
Given  $\varepsilon>0$, consider $L>0$ assigned to $\varepsilon$ by the specification property. For each sequence $\tilde x_0,\dots, \tilde x_n$  in  $\widetilde Y$ and each sequence $0\le t_0<\cdots< t_{n+1}$ such that $t_{i+1}-t_i \geq L$ for all $0\le i\le n$, we have  $x_0,\dots,  x_n$ the unique representatives in  $\Omega_\psi \setminus D$ in the equivalence classes of $x_0,\dots,  x_n$, respectively. Then, as $\psi$
has the (periodic) specification property in  $\Omega_\psi \setminus D$, there  are  a (periodic) point $y \in \Omega_\psi \setminus D$ and  a function $r:\mathbb{R}^+_0 \rightarrow \mathbb{R}^+_0$  constant on each interval $[t_i, t_{i+1}[$  satisfying
$$ r([t_0, t_1[)=0\quad \text{and}\quad
 |r([t_{i+1}, t_{i+2}[) - r([t_i, t_{i+1}[)| < \varepsilon,
$$
for which
$$d\,(\psi_{t+r(t)}(y),\psi_t(x_i)) < \varepsilon, \quad  \forall\, t \in [t_i, t_{i+1}[ \quad\forall \,0\le i \le n.$$
Therefore, the $\widetilde\psi$-orbit of $\tilde y$ is well defined (and periodic) and, using \eqref{eq.dtil}, we get
$$\tilde d\,(\widetilde\psi_{t+r(t)}(\tilde y),\widetilde\psi_t(\tilde x_i))\le d\,(\psi_{t+r(t)}(y),\psi_t(x_i)) < \varepsilon, \quad  \forall\, t \in [t_i, t_{i+1}[ \quad\forall \,0\le i \le n,$$
thus proving that $\widetilde\psi$ has the (periodic) specification property.
\end{proof}

Letting   $i:\Omega_\psi \setminus D\to X$ be the inclusion map and the subscript~$*$  stand for the push-forward of a measure, the next result follows from \cite[Theorem A]{AC14} and \cite[Lemma 4.7]{ACV15}.

\begin{lemma}\label{le.isomorphism}
 The map $i_*: \mathcal M_{ \psi}(\Omega_\psi \setminus D) \to \mathcal M_\psi(X )$ is a bijection.
\end{lemma}

The following lemma is a straightforward consequence of the previous considerations concerning the map $H$ together with the fact that $H^{-1}$ is measurable by  \cite{P66}.

\begin{lemma}\label{le.isomorphism2}
 The map $H_*: \mathcal M_{ \psi}(\Omega_\psi \setminus D) \to \mathcal M_{\widetilde\psi}(\widetilde Y )$ is a bijection.
\end{lemma}

\section{Existence and uniqueness of equilibrium states}

Here we finish the proof of Theorem~\ref{te.existence_uniqueness}.
In the next result we show that    $H$ preserves some good properties of potentials.

\begin{lemma}\label{le.Ves}
If $f \in V^*(\psi)$, then $f\circ H^{-1} \in V(\widetilde\psi)$.
\end{lemma}

\begin{proof}
By definition of $V^*(\psi)$,
we have $f$  continuous and $f(x)=f(I(x))$ for all $x\in D$. Recalling the definition of the equivalence relation $\sim$, we can easily see that $f\circ H^{-1} $ is continuous.
We are left to check that condition 
\eqref{temK} holds.
%

Given  $f\in V^*(\psi)$, there are constants $K>0$ and $\varepsilon > 0$ for which
\eqref{temK*}
holds
whenever
$d\,(\psi_s(x), \psi_s(y)) < \varepsilon$ for all $ s  \in [0,t]$ such that $\psi_s(x),\psi_s(y)\notin B_\varepsilon(D)$.  Consider $\tilde x,\tilde y\in \widetilde Y$ and $t>0$  such that $d\,(\widetilde\psi_s(\tilde x), \widetilde\psi_s(\tilde y)) < \varepsilon/2$ for all $ s  \in [0,t]$. Letting $x=H^{-1}(\tilde x)$ and $y=H^{-1}(\tilde y)$, we have for all $s\in [0,t]$
 \begin{equation}\label{eq.psitil}
 \psi_s(x)=H^{-1}(\widetilde \psi_s(\tilde x))\qand \psi_s(y)=H^{-1}(\widetilde \psi_s(\tilde y)).
\end{equation}
Observe  that the equivalence classes of $\psi_s(x)$ and $\psi_s(y)$ are reduced to a single point. Then, if $\psi_s(x),\psi_s(y)\notin B_\varepsilon(D)$, it follows from Lemma~\ref{le.chain} that
$d\,(\psi_s( x), \psi_s( y)) < \varepsilon.$
Therefore,
\begin{equation}\label{eq.capa}
\left|\int_0^t f(\psi_s(x))\,ds - \int_0^t f(\psi_s(y))\,ds\right|< K.
\end{equation}
Recalling \eqref{eq.psitil}, we have
$$
\left|\int_0^t f(H^{-1}(\widetilde \psi_s(\tilde x)))\,ds - \int_0^t f(H^{-1}(\widetilde \psi_s(\tilde y)))\,ds\right|=\left|\int_0^t f(\psi_s(x))\,ds - \int_0^t f(\psi_s(y))\,ds\right|,
$$
which together with \eqref{eq.capa} gives the desired conclusion.
\end{proof}

Given  a potential $f\in V^*(\psi)$ and considering $\tilde f=f\circ H^{-1}$, it follows from Lemma~\ref{le.Ves} that $\tilde f\in V(\widetilde\psi)$. Therefore, we may use \cite{F77} and obtain an equilibrium state $\tilde\mu_{\tilde f}$ for  $\tilde f$ (with respect to $\widetilde\psi$).
Taking $\mu_f=   i_* H^{-1}_*\tilde\mu_{\tilde f}$, we are going to verify that $\mu_f$ is an equilibrium state for~$f$ (with respect to $\psi$). Firstly notice that it follows from Lemmas~\ref{le.isomorphism} and~\ref{le.isomorphism2}  that for every  $\nu \in  M_\psi(X)$ we have
\begin{equation}\label{eq.cincum}
\int \tilde{f}\, dH_* i_*^{-1}\nu = \int f\circ H^{-1} \,dH_* i_*^{-1}\nu=\int f\,d\, i_*^{-1}\nu=\int f\,d\,\nu,
\end{equation}
and also that, using  \eqref{eq.conjuga2} we get
\begin{equation}\label{eq.cincois}
h_{\tilde\nu}(\widetilde\psi_1)=h_{\nu}\left(\psi_1\vert_{\Omega_\psi \setminus D}\right)=h_{ \nu}\left( \psi_1\right),
\end{equation}
where $\tilde \nu=H_* i_*^{-1}\nu$. From \eqref{eq.cincum} and \eqref{eq.cincois} we deduce that
\begin{align}\label{eq.varprin}
\sup\left\{h_{\nu}(\psi_1) + \int f\,d\nu \colon \, \nu \in M_\psi(X)\right\}&= \sup\left\{h_{\tilde\nu}(\widetilde\psi_1) + \int \tilde f\,d\tilde\nu \colon \nu \in M_\psi(X)\right\}\nonumber \\
 &=\sup\left\{h_{\eta}(\widetilde\psi_1) + \int \tilde f\,d\eta \colon \eta \in M_{\widetilde\psi}(\widetilde Y)\right\},
\end{align}
where the last equality is due to Lemmas~\ref{le.isomorphism} and~\ref{le.isomorphism2}.
As $\tilde\mu_{\tilde f}$ is an equilibrium state for $\tilde f$ and \eqref{eq.cincois} holds for $\nu=\mu_f$, we conclude that
 $$h_{ \mu_{ f}}\left( \psi_1\right)= \sup\left\{h_{\nu}(\psi_1) + \int f\,d\nu \colon \, \nu \in M_\psi(X)\right\}.$$
Additionally, assuming that $\dim(X)<\infty$  and  $\#I^{-1}(\{y\})\le k$ for every $y \in I(D)$, it follows from   \cite[Theorem~3.3.7]{E95} that $\dim(\widetilde Y)<\infty$, and so $\tilde\mu_{\tilde f}$ is unique by \cite[Theorem~2.15]{F77}. Since we have $\mu_f=   i_* H^{-1}_*\tilde\mu_{\tilde f}$, the equilibrium state $\mu_f$ is unique as well.

\section{Variational principle}

According to \cite{F77}, given a continuous semiflow $\varphi$ on a metric space $X$ and a continuous potential $f:X \rightarrow \mathbb{R}$, the topological pressure of $f$ 
may be described thermodynamically as
\begin{equation}\label{eq.pressure}
P(\varphi,f) = \sup_{\mu \in \mathcal{M}_{\varphi}(X)} \,\left\{h_\mu(\varphi_1) + \int f\,d\mu\right\}.
\end{equation}
The aim of this section is to prove a generalization of this equality to impulsive semiflows, replacing the classical notion of topological pressure by the new one (cf. Section~\ref{sse:new-definition}).

Consider a compact metric space $(X,d)$, an impulsive dynamical system  $(X,\varphi,D,I)$ satisfying conditions (C1)-(C5) and $\tau$ the admissible function of the corresponding impulsive times. The assumptions on $I(D)$ ensure that the function that assigns to each $x\in X$ the sequence of visit times to $I(D)$, say $\theta(x):=(\theta_n(x))_{n \in \mathbb{N}}$, is an admissible function with respect to $I(D)$. Moreover, as $I(D)\cap D=\emptyset$, we may re-index the sequences $\tau(x)$ and $\theta(x)$ in order to assemble them in a unique admissible function $\tau'$, 
 where $\tau'_n(x)$ is either $\tau_m(x)$ or $\theta_m(x)$, for some $m$. This way, we have $\tau'\succ \tau$.

\begin{lemma}\label{le.equal1} For every continuous potential $f:X \rightarrow \mathbb{R}$, we have $P^\tau(\psi,f)= P^{\tau'}(\psi,f).$
\end{lemma}

\begin{proof}
Let $f$ be a continuous potential in $X$. As $\tau'\succ \tau$, by Lemma~\ref{le.refine_invariance} we have $$P^{\tau'}(\psi, f)\leq P^\tau(\psi, f).$$ Concerning the other inequality, it follows from the proof of \cite[Lemma 3.1]{ACV15} that, given $t>0$ and small enough $\varepsilon> 0$ and $\delta>0$, then a set $E$ which is finite and $(\psi,\tau',\delta, \varepsilon, t)$-separated  is $(\psi,\tau,\delta, \varepsilon/2, t)$-separated as well. Therefore
$$Z^{\tau}(\psi, \delta, \varepsilon, t) \leq Z^{\tau'}(\psi, \delta,\varepsilon/2,t)$$
and so $P^\tau(\psi,f)\leq P^{\tau'}(\psi,f)$.
\end{proof}

As the distance between the compact sets $D$ and $I(D)$ is strictly positive, fixing $\eta>0$ (recall Remark~\ref{taueta}) and $\xi_0>0$ associated to the conditions (C3)-(C5), we may choose
\begin{equation}\label{eq.xi}
0<\xi<\min\,\{\eta/4, \xi_0/2\}
\end{equation}
small enough so that $I(D)\cap D_\xi=\emptyset$. The next result shows that, with this suitable choice of $\xi$, the $\tau$ and $\tau'$-topological pressures of the semiflows $\psi$ and $\psi_{|_{X_\xi}}$ coincide for potentials in $V^*(\psi)$.

\begin{lemma}\label{le.equal2} If $f \in V^*(\psi)$, then  there is $0<\xi<\xi_0$ such that $P^{\tau'}(\psi,f)=P^{\tau'}(\psi|_{{X_\xi}}, f)$.
\end{lemma}

\begin{proof}
Consider a potential $f \in V^*(\psi)$. As $X_\xi \subset X$, then $P^{\tau'}(\psi|_{{X_\xi}},f)\leq P^{\tau'}(\psi,f).$ We are left to prove the other inequality. It follows from the assumption that $D$ satisfies condition (C3) and the proof of \cite[Lemma 3.2]{ACV15} that, given $t>0$ and small enough $\varepsilon> 0$ and $\delta>0$, if a set $E\subset X$ is finite and $(\psi,\tau',\delta, \varepsilon, t)$-separated and we take the subsets
\begin{equation}\label{eq.EAB}
A= E \cap (D \cup D_\xi)\qand B=E \cap X_\xi,,
\end{equation}
 then $B$ is $(\psi_{\vert X_\xi},\tau',\delta, \varepsilon, t)$-separated and that there exists $\varepsilon' < \varepsilon$ such that $\psi_\xi(A)$ is $(\psi_{\vert X_\xi},\tau',\delta, \varepsilon', t)$-separated. Moreover, $\# \psi_\xi(A)=\# A$. Hence, $\psi_\xi(A) \cup B$ is finite and $(\psi_{\vert X_\xi},\tau',\delta, \varepsilon', t)$-separated.

Since $f \in V^*(\psi)$, there are $\rho>0$ and $K>0$ such that for every $t>0$ we have
$$\left|\int_0^t f(\psi_s(x))\,ds - \int_0^t f(\psi_s(y))\,ds\right|< K,$$
whenever $d\,(\psi_{s}(x), \psi_s(y)) < \rho$ for all $ s \in [0,t]$ such that $\psi_s(x),\psi_s(y)\notin B_{\rho}(D)$. Now, using
Lemma~\ref{le.alfabeta}, we may find $\alpha>0$ such that $\dist(\varphi_s(x),\varphi_u(x))<\rho$ for all $x\in X$ and  $s,u\geq 0$ with $|s-u|<\alpha$. Therefore, choosing  $0<\xi<\alpha$, we have
$$d(\psi_{s}(x), \psi_s(\psi_\xi (x))) =d(\varphi_{s}(x), \varphi_s(\varphi_\xi (x))) < \rho$$
for all $s \in [0,t]$ such that $\psi_s(x),\psi_s(\psi_\xi(x))\notin B_{\rho}(D)$. Hence,
\begin{align*}
\sum_{x \in A} \exp&\left(\int_0^t f(\psi_s(x))ds\right)= \\
&= \sum_{x \in A}\exp\left(\int_0^t f(\psi_s(x))ds - \int_0^t f(\psi_s(\psi_\xi(x)))ds + \int_0^t f(\psi_s(\psi_\xi(x)))ds \right)\\
&\leq  \sum_{x \in A} \exp\left(\left|\int_0^t f(\psi_s(x))ds - \int_0^t f(\psi_s(\psi_\xi(x)))ds \right|\right) \exp\left(\int_0^t f(\psi_s(\psi_\xi(x)))ds\right)\\
&\leq  e^K \sum_{x \in A}\exp\left(\int_0^t f(\psi_s(\psi_\xi(x)))ds\right) = e^K \sum_{y \in \psi_\xi(A)}\exp\left(\int_0^t f(\psi_s(y))ds\right).
\end{align*}
On the other hand, recalling \eqref{eq.EAB} we may write
$$\sum_{x \in E} \exp\left({\int_0^t f(\psi_s(x))ds}\right) = \sum_{x \in A} \exp\left({\int_0^t f(\psi_s(x))ds}\right) + \sum_{x \in B} \exp\left({\int_0^t f(\psi_s(x))ds}\right).$$
Therefore, as $e^K > 1$ and $\varepsilon' < \varepsilon$
we get
$$\sum_{x \in E} \exp\left({\int_0^t f(\psi_s(x))ds}\right) \leq e^K  \sum_{y \in \psi(A) \cup B} \exp\left({\int_0^t f(\psi_s(y))ds}\right),$$
and so
$$Z^{\tau'}(\psi,f,\delta,\varepsilon,t) \leq e^K  Z^{\tau'}(\psi|_{{X_\xi}},f,\delta,\varepsilon',t).$$
This implies that $P^{\tau'}(\psi,f) \leq P^{\tau'}(\psi|_{{X_\xi}}, f)$.
As $I(D)$ also satisfies condition (C3), a similar argument shows that $P^{\tau'}(\psi,f)=P^{\tau'}(\psi|_{{X_\xi \setminus I(D)}}, f)$.
\end{proof}

Using the semiconjugacy $H$ between the semiflows $\psi$ and $\tilde \psi$, we now project on $\widetilde{X}_\xi$ the admissible functions $\tau$ and $\tau'$, as done in \cite[Lemma 4.5]{ACV15}, thus getting admissible functions on $\widetilde{X}_\xi$. This way, we may compare the corresponding pressures.

\begin{lemma}\label{le.equal3} If $f \in V^*(\psi)$, then     $P^{\tau'}(\psi,f) = P^{\tilde\tau'}(\tilde\psi,\tilde f).$
\end{lemma}

\begin{proof}
Given $f\in V^*(\psi)$,
let  $0<\xi<\xi_0$ be given by~Lemma~\ref{le.equal2}.
It follows from  Lemma~\ref{le.invariant} applied to the semiconjugacy $H:X_\xi \to \widetilde{X}_\xi$ that $P^{\tilde\tau'}(\tilde\psi, \tilde f)\leq P^{\tau'}(\psi|_{{X_\xi}}, f)$. Additionally, Lemma~\ref{le.equal2} ensures that $P^{\tau'}(\psi|_{{X_\xi}}, f)=P^{\tau'}(\psi, f)$. Thus,
\begin{equation}\label{eq.other}
P^{\tilde\tau'}(\tilde\psi, \tilde f)\leq P^{\tau'}(\psi, f).
\end{equation}
Conversely, as $H^{-1}:\widetilde{X}_\xi \setminus \pi(D) \rightarrow X_\xi\setminus I(D)$ is a $(\tilde\tau', \tau')$-semiconjugacy between $\tilde\psi\vert_{{\widetilde X_\xi\setminus \pi(D)}}$ and $\psi|_{{X_\xi\setminus I(D)}}$, it follows from Lemma~\ref{le.invariant} that $P^{\tilde\tau'}(\tilde\psi{|_{\widetilde{X}_\xi \setminus \pi(D)}}, \tilde f)\geq P^{\tau'}(\psi|_{{X_\xi \setminus I(D)}}, f).$ Moreover, as $f \in V^*(\psi)$, from Lemma~\ref{le.equal2} we have $P^{\tau'}(\psi|_{{X_\xi \setminus I(D)}}, f)=P^{\tau'}(\psi, f)$. Besides, as $\widetilde{X}_\xi \setminus \pi(D) \subset \widetilde{X}_\xi$, we also know that $P^{\tilde\tau'}(\tilde\psi, \tilde f)\geq P^{\tilde\tau'}(\tilde\psi{|_{\widetilde{X}_\xi \setminus \pi(D)}}, \tilde f)$. Hence,
$$P^{\tilde\tau'}(\tilde\psi, \tilde f)\geq P^{\tilde\tau'}(\tilde\psi{|_{\widetilde{X}_\xi \setminus \pi(D)}})\geq P^{\tau'}(\psi|_{{X_\xi \setminus I(D)}}, f)=P^{\tau'}(\psi, f),$$
which together with~\eqref{eq.other} gives the result. \end{proof}

Let us resume the proof of Theorem~\ref{te.pressure}. Given $f\in V^*(\psi)$, it follows from
Theorem~\ref{te.coincide} that
$$P^{\tilde\tau'}(\tilde\psi,\tilde f)=P(\tilde\psi, \tilde f),$$
which together with Lemma~\ref{le.equal1}, Lemma~\ref{le.equal2} and Lemma~\ref{le.equal3}
 yields
$$P^\tau(\psi,f) = P(\tilde\psi,\tilde f).$$
Now, it follows from  \eqref{eq.varprin}
 that
$$\sup_{\eta \in \mathcal{M}_{\tilde\psi}(X)} \,\left\{h_\eta(\tilde\psi_1) + \int \tilde f\,d\eta\right\} = \sup_{\nu \in \mathcal{M}_{\psi}(X)} \,\left\{h_\mu(\psi_1) + \int f\,d\nu\right\}.$$
Moreover, equation \eqref{eq.pressure} gives
$$P(\tilde\psi,\tilde f) = \sup_{\eta \in \mathcal{M}_{\tilde\psi}(X)} \,\left\{h_\eta(\tilde\psi_1) + \int \tilde f\,d\eta\right\}.$$
Therefore,
$$P^\tau(\psi,f) = \sup_{\nu \in \mathcal{M}_{\psi}(X)} \,\left\{h_\mu(\psi_1) + \int f\,d\nu\right\},$$
which ends the proof of Theorem~\ref{te.pressure}.

\section{Examples}\label{se.examples}

Here we give two examples of impulsive dynamical systems for which the variational principle and the existence and uniqueness of equilibrium states hold. This follows from the simple fact that the impulsive semiflow is uniquely ergodic in the first example. In the second one we show that the impulsive semiflow is expansive and has the specification property and then use Theorem~\ref{te.existence_uniqueness} and Theorem~\ref{te.pressure}.


\subsection{Suspension of  a rotation} Here we define an impulsive semiflow on a suspension of an irrational rotation on $S^1$ which is uniquely ergodic. Consider the unit circle $S^1=\{e^{2\pi i x} \in \mathbb{C} \colon 0\leq x < 1\}$ and an irrational number  $\theta_1 \in \,]0,1[$. Let $R_{\theta_1}: S^1 \rightarrow S^1$ be the irrational rotation $R_{\theta_1}(e^{2\pi i x})=e^{2\pi i (x+\theta_1)}$ on $S^1$. Consider the cylinder
\begin{equation*}\label{eq.Yc1}
 Y=\{(z,u)\colon z\in S^1, \,\, 0\le u\le 1\}
\end{equation*}
and the $2$-torus
\begin{equation*}\label{eq.Xc1}
 X=Y/{\approx},
\end{equation*}
where $\approx$ is the equivalence relation in $X$ given by $(z,1)\approx (R_{\theta_1}(z),0).$ We define the suspension flow $\varphi:\mathbb R^+_0\times X\to X$ over the rotation $R_{\theta_1}$ as
 $$
 \varphi(t,(z,u))=
 \begin{cases}
 (z,t+u) ,& 0\le t+u < 1;\\
 (R_{\theta_1}(z),0), & t+u=1.
 \end{cases}
 $$
As $R_{\theta_1}$ is uniquely ergodic, the unique invariant probability measure invariant by the suspension flow $\varphi$ is Lebesgue measure on $X=S^1\times S^1$; see \cite[Chapter 6]{PP90}.

Now, take another irrational number $\theta_2 \in \, ]0,1[$ and  consider the set
 $$D=S^1\times\left\{\frac{1}{2}\right\}\subset X.$$
 Let the impulsive function $I:D\to X$ be the map defined as $$I\left(z,\frac{1}{2}\right)=\left(R_{\theta_2}(z),\frac{3}{4}\right),$$ where $R_{\theta_2}$ is the irrational rotation of angle $\theta_2$. Then define $\psi:\mathbb R^+_0\times X\to X$ as the impulsive semiflow of the impulsive dynamical system $(X,\varphi, D, I)$.  It is straightforward to check that the impulsive semiflow $\psi$ satisfies  conditions (C1)-(C5).

 Consider the suspension flow $\widetilde{\varphi}$ on $X$ over the rotation $R_{\theta_1+\theta_2}$; as in the case of $\varphi$, the flow $\widetilde{\varphi}$ is uniquely ergodic. We note that the map $\mathcal{F}: \Omega_\psi \setminus D\to \widetilde X$ defined as
 $$
 \mathcal{F}(x,u)=
  \begin{cases}
 (x,u) ,& 0\le u\le \frac{1}{2};\\
 (R_{\theta_2}^{-1}(x),2u-1), & \frac{3}{4}\le u \le 1;
 \end{cases}
$$
is a continuous bijection. 
Moreover, $\mathcal F$  conjugates the flows $\widetilde{\varphi}$ and $\psi_{\vert_{\Omega_\psi \setminus D}}$, that is
$$\widetilde \varphi_t \circ \mathcal{F}= \mathcal{F} \circ \psi_t$$
for every $t \geq 0$. Consequently, the semiflow $\psi$ is uniquely ergodic and so it has a unique equilibrium state for any continuous potential.

\subsection{Suspension of a shift} Here we define an impulsive semiflow on a suspension of a shift which is expansive and has the specification property.
Let  $(\Sigma_2 , \sigma)$ be the two-sided full shift on 2 symbols. Given two irrational numbers $a,b>3$  linearly independent over $\mathbb Q$, let
$  c:\Sigma_2\to\mathbb R^+$  be the ceiling function defined as $c(x)=a$ if $x_0=0$  and $c(x)=b$ if $x_0=1$. Let
\begin{equation}\label{eq.Yc2}
 Y_c=\{ (x,u)\colon x\in\Sigma_2, \,\, 0\le u\le c(x)\}
\end{equation}
 and
 \begin{equation}\label{eq.Xc2}
 X_c=Y_c/{\sim_c}
 \end{equation}
 where $\sim_c$ is the equivalence relation in $X_c$ given by $(x,c(x))\sim_c (\sigma x,0).$ We define the suspension flow $\varphi_c:\mathbb R^+_0\times X_c\to X_c$ as
 $$
 \varphi_c(t,(x,u))=
 \begin{cases}
 (x,t+u) ,& 0\le t+u\le c(x);\\
 (\sigma x,0), & t+u=c(x).
 \end{cases}
 $$
As $\sigma$ is an expansive map, it follows from \cite[Theorem 6]{BW72} that $\varphi_c$ is an expansive semiflow. It is straightforward to check that $c$ is not cohomologous to a function taking values in $\beta\mathbb Z$ for some $\beta>0$. Therefore, by \cite[Proposition~5]{QS12} the flow $\varphi_c$ has the periodic specification property.
Consider
$$D=\Sigma_2\times\{1\}\subset X_c$$
and the impulsive function $I:D\to X_c$ defined as $$I((x_n)_n,1)=((1-x_n)_n,3).$$ Notice that, on the first coordinate, the map $I$ acts as the isometry $\mathcal{R}: \Sigma_2 \to \Sigma_2$ that reverses each digit,  so condition (C1) holds. Moreover, $Y_c \subset \Sigma_2 \times [0, \max\{a,b\}]$ is a finite dimensional metric space; see \cite[Theorems 4.1.7, 4.1.21 \& 4.1.25]{E95}.
Since each class of the equivalence relation $\sim_c$ has at most two elements, by \cite[Theorem 3.3.7]{E95} the dimension of $X_c$ is finite as well.

Let $\psi_c:\mathbb R^+_0\times X_c\to X_c$ be the impulsive semiflow of the impulsive dynamical system $(X_c,\varphi_c, D,I) $.
 As
$$\Omega_{\psi_c} \setminus D = \Big(\{(x,u)\colon x\in\Sigma_2, \,\, 0\le u < 1\}\cup \{(x,u)\colon x\in\Sigma_2, \,\, 3\le u \le c(x)\}\Big)/{\sim_c},$$
condition (C2) is also valid. Besides, the semiflow $\psi_c$ satisfies conditions (C3)-(C5) for any $0<\xi < \xi_0=2$. We are left to verify that $\psi_c$ restricted to $\Omega_{\psi_c} \setminus D$ is expansive and has the periodic specification property.

 Consider the ceiling function $\tilde c= c \circ \mathcal{R}-2$ and the corresponding spaces $Y_{\tilde c}$ and $X_{\tilde c}$ defined as in \eqref{eq.Yc2} and \eqref{eq.Xc2}, respectively. The suspension flow $\varphi_{\tilde c}$ on $X_{\tilde c}$ is expansive, as in the case of $\varphi_c$; moreover, as $\mathcal{R}$ commutes with $\sigma$ and is an involution, it is not difficult to show that $\tilde c$ is not cohomologous to a function taking values in $\beta\mathbb Z$ for some $\beta>0$, and so the flow $\varphi_{\tilde c}$ has the specification property.

 Let $\mathcal{F}: \Omega_{\psi_c} \setminus D \to X_{\tilde c}$ be defined as
$$
  \mathcal{F}(x,u)=
  \begin{cases}
 (x,u) ,& 0\le u< 1;\\
 (\mathcal{R}^{-1}(x),u-2), & 3\le u\le c(x).
 \end{cases}
$$
It is easy to verify that the map $\mathcal{F}$ is a continuous bijection.
Besides, as $I(x,1)=(\mathcal{R}(x),3)$ for all $x \in \Sigma_2$, the map $\mathcal{F}$ conjugates the semiflows $\psi_c$ on $\Omega_{\psi_c} \setminus D$ and $\varphi_{\tilde c}$. 
Hence, $\psi_c$ is expansive and has the periodic specification property in $\Omega_{\psi_c} \setminus D$. Therefore, by Theorem~\ref{te.existence_uniqueness} the semiflow $\psi$ has a unique equilibrium state for any potential in $V^*(\psi)$.


\subsection*{Acknowledgements} We are grateful to A. Quas and V. Climenhaga for valuable discussions.

\end{document}